\documentclass{article}
\usepackage[utf8]{inputenc}
\usepackage[margin=1in]{geometry}

\usepackage{amsfonts,amsmath,amssymb,amsthm}
\usepackage{xcolor}
\usepackage{graphicx}

\graphicspath{ {./images/} }
\newtheorem{Theorem}{Theorem}[section]
\newtheorem{Lemma}[Theorem]{Lemma}
\newtheorem{Observation}[Theorem]{Observation}
\newtheorem{Corollary}[Theorem]{Corollary}

\theoremstyle{definition}
\newtheorem{Definition}[Theorem]{Definition}
\newcommand{\N}{\mathbb N}
\newcommand{\Q}{\mathbb Q}
\newcommand{\R}{\mathbb R}
\newcommand{\ff}{\mathcal F}
\newcommand{\fuse}{\sim}

\usepackage{algpseudocode}

\title{Ordinals and recursively defined functions on the reals}
\author{Gabriel Nivasch\thanks{Department of Computer Science, Ariel University, Ariel, Israel. \texttt{gabrieln@ariel.ac.il}}, Lior Shiboli\thanks{Department of Computer Science, Ariel University, Ariel, Israel. \texttt{lior12sh@gmail.com}. Research was supported by ISF grant 1065/20.}}
\date{}

\begin{document}

\maketitle
\begin{abstract}
    We determine sufficient conditions under which certain recursively defined functions are well defined for all real inputs. Given a function $f:\mathbb R\to\mathbb R$, call a decreasing sequence $x_1>x_2>x_3>\cdots$ \emph{$f$-bad} if $f(x_1)>f(x_2)>f(x_3)>\cdots$, and call the function $f$ \emph{ordinal decreasing} if there exist no infinite $f$-bad sequences. We prove the following result: Given ordinal decreasing functions $f,g_1,\ldots,g_k,s$ that are everywhere larger than $0$, define the recursive algorithm ``$M(x)$: if $x<0$ return $f(x)$, else return $g_1(-M(x-g_2(-M(x-\cdots-g_k(-M(x-s(x)))\cdots))))$". Then $M(x)$ halts and is ordinal decreasing for all $x \in \mathbb{R}$. The recursive algorithms $M$ and $M_n$ previously studied in the context of fusible numbers by Ericskon et al.~(2022) and Bufetov et al.~(2024), respectively, are special cases of this scheme.

    Moreover, given an ordinal decreasing function $f$, denote by $o(f)$ the ordinal height of the root of the tree of $f$-bad sequences. Then we prove that, for $k\ge 2$, the function $M(x)$ defined by the above algorithm satisfies $o(M)\le\varphi_{k-1}(\gamma+o(s)+1)$, where $\gamma$ is the smallest ordinal such that $\max\{o(s),o(f),o(g_1), \ldots,\allowbreak o(g_k)\} <\varphi_{k-1}(\gamma)$.

\end{abstract}
\section{Introduction}
The \emph{fusible numbers} are a set of real numbers, motivated by a mathematical riddle involving fuses. Fusible numbers were defined by Erickson \cite{Eri} and further studied by Erickson, Nivasch and Xu \cite{ENX,Xu12}. In their research on fusible numbers, they examined a peculiar recursive algorithm $M$, which is the main motivation for this paper.

Algorithm $M$ accepts a real number as input, and returns a real number as output. (Given a rational input, algorithm $M$ returns a rational output.) It is defined recursively as follows:
\begin{equation}\label{eq_M}
        M(x) = 
            \begin{cases}
                -x ,& \text{if  $x < 0$;}\\
                \frac{M(x-M(x-1))}{2}, & \text{if  $x \geq 0$.}
            \end{cases}
\end{equation}
In other words, (\ref{eq_M}) defines a function $M:\mathbb R\to\mathbb R$.
For example, it can be checked that $M(1)=\frac{M(1-M(0))}{2}=\frac{1}{8}$, $M(\frac{3}{2})=\frac{1}{32}$, $M(2)=2^{-10}$, $M(\frac{5}{2})=2^{-31}$ and $M(3)=2^{-1,541,023,937}$. Figure \ref{fig:Mgraph} shows a partial plot of the function $M$.

It is not immediately obvious that $M$ is well defined for all real inputs, meaning that for every input the recursion terminates in a finite number of steps. Erickson, Nivasch and Xu \cite{Eri,ENX,Xu12} proved that in fact $M$ terminates on all real inputs. Furthermore, they proved that Peano Arithmetic cannot prove that $M$ terminates on all rational inputs, nor even on all natural inputs. Hence, in a sense, the termination of algorithm $M$ is indeed not so easy to prove. The PA-independence result was obtained by showing that $\frac{1}{M(n)}$ grows as fast as $F_{\varepsilon_0}(n-7)$ for integers $n\geq 8$. See Section \ref{subsec_fusible} below for background on fusible numbers and their relation to algorithm $M$.

\begin{figure}
            \centering
            \includegraphics[scale=0.3]{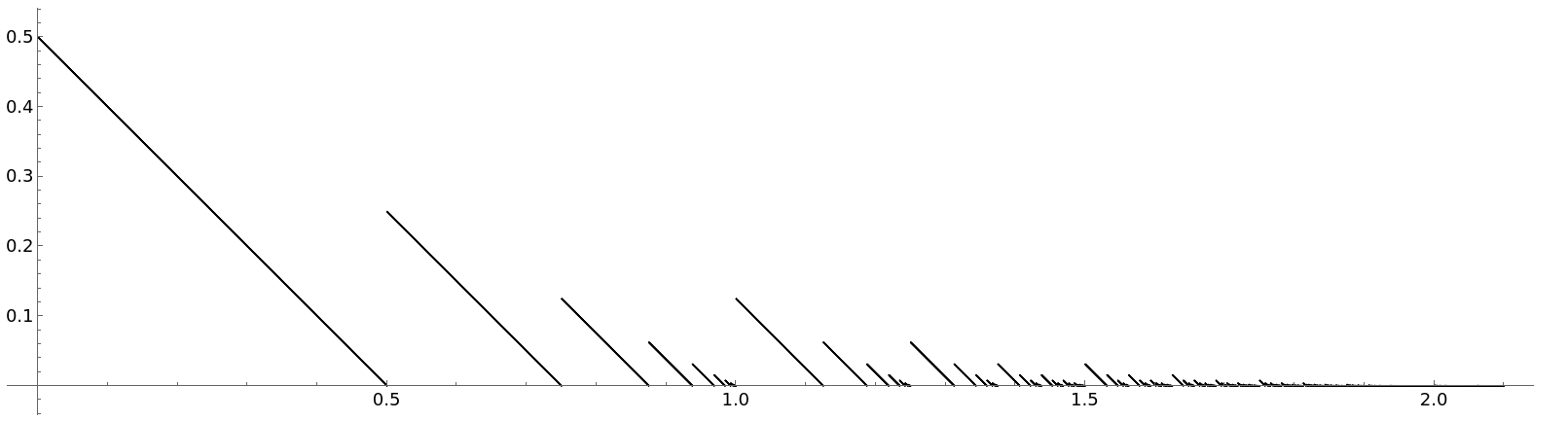}
            \caption{The graph of $M(x)$ up to $x=2.1$.}
            \label{fig:Mgraph}
\end{figure}
In a follow-up paper, Bufetov, Nivasch and Pakhomov \cite{BNP} studied a generalization of fusible numbers to \emph{$n$-fusible numbers} and a corresponding generalization of algorithm $M$ to the following algorithm $M_n$:
\begin{equation}
        M_n(x) = 
        \begin{cases}
			-x, & \text{if  $x < 0$;}\\
            \frac{M_n(x-M_n(x-\cdots -M_n(x-1)\cdots))}{n}, & \text{if  $x \geq 0$;}
		  \end{cases}
        \end{equation}
where $M_n(x-\cdots)$ repeats $n$ times. In particular, algorithm $M$ of (\ref{eq_M}) equals $M_2$. Bufetov et al.~showed that for every $n \geq 1$, $M_n$ terminates on all real inputs.

In this paper we study the following question: Which modifications can be done to algorithm $M$, or to its generalization $M_n$, such that they will still halt on all real inputs? For example, what would happen if, in (\ref{eq_M}), we change the denominator $2$ to $3$? What would happen if we change the ``${}-1$'' to something depending on $x$?

In this paper we identify a large class of algorithms that generalize $M_n$ and halt on all real inputs. The description of the algorithms, as well as the proof  that they halt on all real inputs, involve a property of real-valued functions, which we call \emph{ordinal decreasing}.

\subsection{Ordinal decreasing functions}

\begin{Definition}\label{def_ord_dec}
    Given a function $f:D\to \mathbb{R}$ for some $D \subseteq \mathbb{R}$, we call a descending sequence $x_1>x_2>x_3>\cdots$ in $D$ \emph{$f$-bad} if $f(x_1)>f(x_2)>f(x_3)>\cdots$. We call the function $f$ \emph{ordinal decreasing} if there exist no infinite $f$-bad sequences.
\end{Definition}

Some examples of ordinal decreasing functions are: 
\begin{enumerate}
    \item Every nonincreasing function.
    \item A function $f$ in the domain $[0,1)$ that is divided into a sequence of intervals $\{I_n\}$ where $I_n = [1-2^{-n},1-2^{-(n+1)})$ and where $f$ is decreasing in each interval (see Figure \ref{fig:graph2}, left).
    \item A function $f$ in the domain $(0,1]$ that is divided into a sequence of intervals $\{I_n\}$ where $I_n = (2^{-(n+1)},2^{-n}]$, such that $f$ is decreasing in each $I_n$ but such that for each $n$ there exist only finitely many $m>n$ such that $f(2^{-m})<\lim_{x \to (2^{-n})^+}(f(x))$ (see Figure \ref{fig:graph2}, right).
\end{enumerate}
\begin{figure}
    \centering
    \includegraphics[scale=0.3]{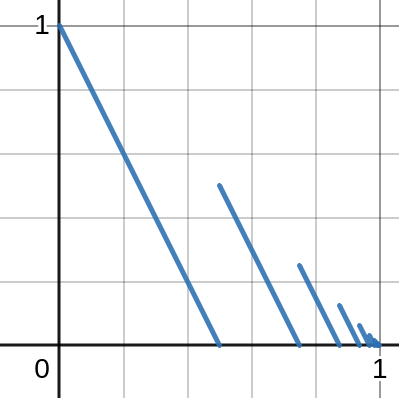}
    \includegraphics[scale=0.3375]{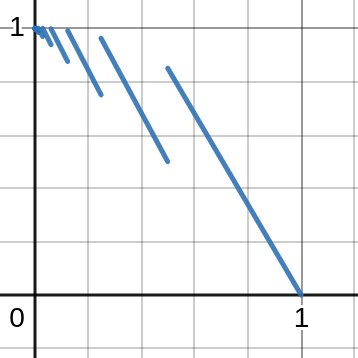}
    \caption{Examples of ordinal decreasing functions.}
    \label{fig:graph2}
\end{figure}

The notion of an $f$-bad sequence is analogous to the notion of a bad sequence in a well partial order (see Section \ref{subsec_WPO} below for background on well partial orders). We can naturally extend the analogy, and define the tree of $f$-bad sequences, as well as the ordinal type $o(f)$ of the function:

\begin{Definition}
    Given an ordinal decreasing function $f$, the \emph{tree of $f$-bad sequences} $T_f$ is the (possibly infinite) tree that contains a vertex $v_f(\overline x)$ for each $f$-bad sequence $\overline x = \langle x_1>\cdots>x_n\rangle$, and contains an edge connecting each $v_f(\langle x_1>\cdots>x_n\rangle)$, $n\ge 1$ to its parent $v_f(\langle x_1>\cdots>x_{n-1}\rangle)$. The \emph{root} of the tree is $v_f(\langle\rangle)$, corresponding to the empty sequence.
\end{Definition}

Since the tree $T_f$ contains no infinite path, there exists a unique way to assign to each vertex $v\in T_f$ an ordinal height $o(v)$, such that $o(v)=\sup_{\text{$w$ child of $v$}} (o(w)+1)$ for all $v\in T_f$.

\begin{Definition}\label{def_ODF_type}
The \emph{ordinal type} $o(f)$ of the ordinal decreasing function $f$ is the ordinal height of the root of $T_f$, meaning $o(f) = o(v_f(\langle\rangle))$. Given an interval $D\subset \mathbb R$, we denote by $o(f|_D)$ the ordinal type of the restriction of $f$ to $D$. If $D=(-\infty,a)$ then we also write $o(f|_D)$ as $o(f|_a)$.
\end{Definition}

\begin{Definition}
    Let $f$ be an ordinal decreasing function. Then $o_f$ is the function from the reals to the ordinals recursively given by $o_f(x)= \sup_{x'<x, f(x') < f(x)}(o_f(x')+1)$.
\end{Definition}

It is easy to verify that  $o(f)=\sup_{x\in \mathbb{R}}(o_f(x)+1)$.

\subsection{Our results}
    Now we can state the main results of our paper.
    \begin{Theorem}
    \label{M halts}
    Consider the recursive algorithm:
        \begin{equation}
        M(x) = 
            \begin{cases}
                f(x)& \text{if  $x < 0$;}\\
                g_1(-M(x-g_2(-M(x-\cdots-g_k(-M(x-s(x)))\cdots)))) & \text{if  $x \geq 0$;}
            \end{cases}
        \end{equation}
    where the functions $s(x)$, $f(x)$ and $g_i(x)$ for all $i$ are all ordinal decreasing and larger than 0 for every $x$ in the appropriate ranges: $(-\infty, 0)$ for $f,g_i$, and $[0,\infty)$ for $s$.
    
    Then $M(x)$ halts and is ordinal decreasing for all $x \in \mathbb{R}$. 
    \end{Theorem}
    Note that Theorem \ref{M halts} covers the cases mentioned above, by taking $k=n$, $f(x)=-x$, $g_1(x)=-\frac{x}{n}$, $g_i(x)=-x$ for $2 \leq i \leq n$ and $s(x)=1$.

    We also prove the following upper bounds on $o(M)$ in terms of $k$ and $o(f), o(s), o(g_1),\ldots, o(g_k)$.

    \begin{Theorem}\label{thm_o_bound}
        Let $M$ be the function computed by the algorithm of Theorem \ref{M halts}. If $k=1$, then let $\gamma$ satisfy $\max{\{o(f), o(s), o(g_1)\}} < \omega^{\omega^\gamma}$. Then $o(M)\le \omega^{\omega^{\gamma+1}(o(s)+1)}$. For $k\ge 2$, let $\gamma$ satisfy
        $\max\{o(f), o(s), o(g_1), \ldots,\allowbreak o(g_k)\} < \varphi_{k-1}(\gamma)$. Then $o(M)\le\varphi_{k-1}(\gamma+o(s)+1)$.
    \end{Theorem}

    By comparison, the specific function $M$ of Erickson et al.~\cite{ENX} satisfies $o(M)=\varphi_1(0)=\varepsilon_0$, and the generalization of Bufetov et al.~\cite{BNP} satisfies $o(M)=\varphi_{n-1}(0)$. (See Section \ref{subsec_veblen} below for the definition of the $\varphi$ notation.) 

\section{Background}

\subsection{Fusible numbers}\label{subsec_fusible}
In this section we give background on fusible numbers and their generalizations, and their relation to algorithms $M$ and $M_n$. Strictly speaking, this background is not necessary to understand the results of this paper.

Consider the following riddle: We have available two fuses, each of which
will burn for one hour when lit. How can we use the two fuses to measure 45 minutes? The solution is to light both ends of one fuse simultaneously, and at the same time light
one end of the second fuse. The first fuse will burn out after 30 minutes. At that moment we
light the second end of the second fuse. The second fuse will burn out after an additional 15
minutes, bringing the total time elapsed to 45 minutes.

Based on this puzzle, Erickson~\cite{Eri} defined \emph{fusible numbers} as the set $\ff\subset\Q$ of all times that can be similarly measured by using any number of fuses that burn for unit time. Formally, he defined 
\begin{equation*}
x\fuse y=(x+y+1)/2\qquad\text{for $x,y\in \R$.}
\end{equation*}
If $|y-x|<1$,  then the number $x\fuse y$ (read ``$x$ fuse $y$'') represents the time at which a fuse burns out, if its endpoints are lit at times $x$ and $y$, respectively. Then $\ff$ is defined recursively by $0\in\ff$, and $x\fuse y\in\ff$ whenever $x,y\in\ff$ with $|y-x|<1$.

As Erickson et al.~\cite{ENX} showed, $\ff$ is a well-ordered subset of $\mathbb R$, with ordinal type $o(\ff)=\varepsilon_0$. In order to prove the lower bound on the ordinal type, they considered a subset $\ff'\subset\ff$, which they called \emph{tame fusible numbers}.

The set $\ff'$ is defined from the bottom up by transfinite induction. At each stage of the induction, an initial segment $\mathcal H$ of the current $\ff'$ is marked as ``used''. At every step of the construction process, if $x = \min(\ff'\setminus\mathcal H)$ is the smallest unused element, then $\ff'\cap[0,x+1)$ has already been defined.

We start by letting $\ff'\cap [0,1)=\{1-2^{-n}:n\in\N\}$. All these numbers are fusible, since $0\fuse (1-2^{-n})=1-2^{-(n+1)}$. We also initialize $\mathcal H=\emptyset$; meaning, we have not yet used any element of $\ff'$. 

Now we repeatedly do the following: Let $x =\min (\ff'\setminus\mathcal H)$ be the smallest unused element, and suppose by induction that $\ff'\cap[0,x+1)$ has already been defined. Let $y$ be the smallest element of $\ff'$ that is strictly larger than $x$, and let $m=y-x$. We define $\ff'\cap [x+1,y+1)$ by taking the interval $\ff'\cap I_0$ for $I_0 = [x+1-m,x+1)$ (which has already been defined), and applying to it ``$y\fuse{}$'', obtaining $\ff'\cap I_1$ for $I_1 = [x+1,x+1+m/2)$, and then applying  ``$y\fuse{}$'' to this latter set to obtain $\ff'\cap I_2$ for $I_2 = [x+1+m/2,x+1+3m/4)$, and so on. In general, denoting $I_n = [y+1-m/2^{n-1},y+1-m/2^n)$ for $n\in\N$, each interval $\ff'\cap I_{n+1}$ is a scaled-down copy of $\ff'\cap I_n$ obtained by applying ``$y\fuse{}$''. See Figure~\ref{fig_ENX_tame}. Finally, we add $x$ to $\mathcal H$ and repeat.

\begin{figure}
\centering{\includegraphics{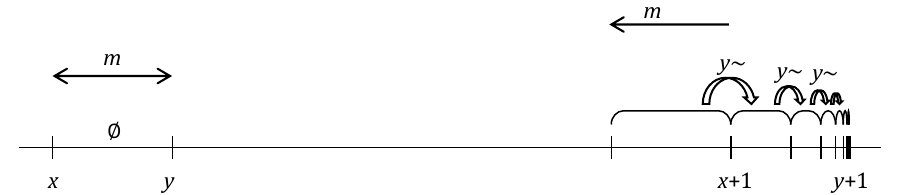}}
\caption{\label{fig_ENX_tame}Construction of the tame fusible numbers in \cite{ENX}, which motivate their algorithm $M$.}
\end{figure}

At the end of the transfinite induction, $\ff'\cap[0,\infty)=\ff'$ is defined; and $\mathcal H = \ff'$, meaning all the elements of $\ff'$ were used.

The authors in \cite{ENX} proved that the ordinal type of $\ff'$ is $o(\ff')=\varepsilon_0$, by first proving the following property by ordinal induction: For every $x\in\ff'$, if $o(\ff'\cap[0,x))=\alpha$, then $o(\ff'\cap[0,x+1))=\omega^\alpha$. Indeed, assume by ordinal induction that $o(\ff'\cap[0,x))=\alpha$ and that $o(\ff'\cap[0,x+1))=\omega^\alpha$. Let $y=\min(\ff'\cap(x,\infty))$ be the successor of $x$ in $\ff'$. On the one hand, $o(\ff'\cap[0,y))=\alpha+1$. On the other hand, $\ff'\cap[x+1,y+1)$ consists of $\omega$-many copies of a suffix of $\ff'\cap[0,x+1)$ (see Figure \ref{fig_ENX_tame} again). This suffix also has ordinal type $\omega^\alpha$. Hence, $o(\ff'\cap[0,y+1))=\omega^\alpha+\omega^\alpha\cdot\omega=\omega^{\alpha+1}$, as claimed.

Hence, the authors in \cite{ENX} concluded, for $n\in\N$ we have $o(\ff'\cap[0,n))= \omega^{\omega^{\cdots^\omega}}$ with $n$ $\omega$'s, and so $o(\ff')=\varepsilon_0$.

\paragraph{\boldmath The algorithm $M$.} The significance of algorithm $M$ for tame fusible number is as follows: For each $x\in \R$, $M(x)$ equals the distance between $x$ and the smallest tame fusible number strictly larger than $x$. Indeed, for $x<0$ we have $M(x)=-x$, and for $x\ge 0$ we recursively have $M(x)=M(x-M(x-1))/2$, as shown in \cite{ENX}. In other words, the set $\{x+M(x):x\in\R\}$ is actually a countable set, which equals $\ff'$. The graph of $M(x)$ consists of line segments with slope $-1$ approaching (without actually touching) the elements of $\ff'$ at the $x$-axis (see Figure \ref{fig:Mgraph} again).

Hence, $M$ is an ordinal decreasing function according to our Definition \ref{def_ord_dec}. Since $o(\ff')=\varepsilon_0$, it follows easily that the ordinal type of $M$ (according to our Definition \ref{def_ODF_type}) is $o(M)=\varepsilon_0$ as well.

\paragraph{Generalized fusible numbers.} Bufetov et al.~\cite{BNP} generalized fusible numbers as follows: Suppose we have a finite collection of water tanks, where each water tank has $n$ faucets. Each faucet, on its own, can empty the water tank in $1$ hour. A faucet can only be opened at time $0$ or at the moment another tank has become completely empty. We are interested in the set of times $t$ for which there is a way to make a tank empty precisely at time $t$ according to these rules.

A tank whose faucets are opened at times $x_1, \ldots, x_n$ will empty at time $g(x_1,\ldots,x_n)=(x_1 + \cdots x_n + 1)/n$. Hence, the set $\mathcal F_n$ of \emph{$n$-fusible numbers} is obtained by letting $0\in\mathcal F_n$, and letting $g(x_1,\ldots,x_n)\in\mathcal F_n$ whenever $x_1,\ldots,x_n\in \mathcal F_n$ and $g(x_1,\ldots,x_n)>\max\{x_1,\ldots,x_n\}$.

Consider the recursive algorithm $M_n$ given by $M_n(x)=-x$ for $x<0$, and $M_n(x)=M_n(x-M_n(x-\cdots-M_n(x-1)\cdots))/n$ for $x\ge 0$, where $M_n(x-\cdots)$ repeats $n$ times. As Bufetov et al.~showed, the set $\{x+M_n(x):x\in\R\}$ is a subset of $\mathcal F_n$, of order type $\varphi_{n-1}(0)$. See \cite{BNP} for more details.
    
\subsection{Real induction}
In this paper we will use the following result, which is called \emph{real induction} (see Clark \cite{L.Clark} for a survey).
\begin{Lemma}
\label{real induction}
Let $S \subset \mathbb{R}$ be a set that satisfies:

(R1) There exists $a \in \mathbb{R} $ such that $(-\infty,a) \subset S $.

(R2) For all $x \in \mathbb{R}$, if $(-\infty, x) \subset S$, then $x \in S$.

(R3) For all $x \in S $, there exists $y > x$ such that $(x, y) \subset S$. 

Then $S = \mathbb{R}$. 
\end{Lemma}
\begin{proof}
Suppose $S \neq \mathbb{R}$. Let $a= \inf(\mathbb{R}\setminus S)$. By (R1) $a\neq - \infty$. Therefore by (R2), $a \in S$. Therefore (R3) yields a contradiction.
\end{proof}

It is worth noting for our purposes that since  Peano Arithmetic is built upon the natural numbers, we cannot use real induction within Peano Arithmetic, but must rely on Second Order Arithmetic.

\subsection{Veblen functions}\label{subsec_veblen}

    The finite Veblen functions $\varphi_n$, $n\in\mathbb N$ are a sequence of functions from ordinals to ordinals, defined by starting with $\varphi_0(\alpha) = \omega^\alpha$, and for each $n\in\mathbb N$, letting
    \begin{align*}
    \varphi_{n+1}(0)&=\sup_{k\in\mathbb N}\varphi_n^{(k)}(0);\\
    \varphi_{n+1}(\alpha+1)&=\sup_{k\in\mathbb N}\varphi_n^{(k)}(\varphi_{n+1}(\alpha)+1);\\
    \varphi_{n+1}(\alpha) &=\sup_{\beta<\alpha} \varphi_{n+1}(\beta),\qquad\text{$\alpha$ limit}.
    \end{align*}
    Here $f^{(k)} = \underbrace{f\circ f\circ \cdots\circ f}_k$ denotes $k$-fold application of $f$.

    Each function $\varphi_n$, $n\ge 1$ enumerates the fixed points of $\varphi_{n-1}$. In particular, we have $\varphi_{n-1}(\varphi_n(\alpha))=\varphi_n(\alpha)$; and, more generally, for every $n<m$ we have $\varphi_n(\varphi_m(\alpha))=\varphi_m(\alpha)$.
    
    Ordinals of the form $\varphi_1(\alpha)$ are called \emph{epsilon numbers}, and are denoted $\varepsilon_\alpha = \varphi_1(\alpha)$. They satisfy the following property: Let $\tau$ be an ordinal, and let $\gamma$ be maximal such that $\varepsilon_\gamma\le \tau$. Then the next epsilon number $\varepsilon_{\gamma+1}$ can be obtained by an infinite exponential tower of $\tau$, meaning $\varepsilon_{\gamma+1}=\sup\{\tau, \tau^\tau, \tau^{\tau^\tau},\ldots\}$.

    Furthermore, let $n\in\N$, and suppose $\varphi_{n+1}(\gamma)<\tau<\varphi_{n+1}(\gamma+1)$. Then $\varphi_{n+1}(\gamma+1)$ can be obtained by $\omega$ many applications of $\varphi_n$ on $\tau$, meaning $\varphi_{n+1}(\gamma+1)=\sup\{\tau, \varphi_n(\tau), \varphi_n(\varphi_n(\tau)),\allowbreak\varphi_n(\varphi_n(\varphi_n(\tau))),\ldots\}$.

\subsection{Natural sum and product of ordinals}\label{subsec_natural}
Given ordinals $\alpha,\beta$ with Cantor Normal Forms 
\begin{align*}
\alpha &=\omega^{\alpha_1}+\ldots+\omega^{\alpha_n},\qquad\text{ with $\alpha_1\ge\ldots\ge \alpha_n$},\\
\beta &=\omega^{\beta_1}+\ldots+\omega^{\beta_m},\qquad\text{ with $\beta_1\ge\ldots\ge\beta_m$};
\end{align*}
their \emph{natural sum} $\alpha\oplus \beta$ is given by $\omega^{\gamma_1}+\ldots+\omega^{\gamma_{n+m}}$, where $\gamma_1,\ldots,\gamma_{n+m}$ are $\alpha_1,\ldots,\alpha_n,\beta_1,\ldots,\beta_m$ sorted in nonincreasing order. The \emph{natural product} of $\alpha,\beta$ is given by
\begin{equation*}
\alpha\otimes \beta=\mathop{\bigoplus\limits_{1\le i\le n}}\limits_{1\le j\le m}\omega^{\alpha_i\oplus \beta_j}.
\end{equation*}
(See e.g.~de Jongh and Parikh \cite{dejongh_parikh}.)

The natural sum and natural product operations are commutative and associative, and natural product distributes over natural sum. These operations are also monotonic, in the sense that if $\alpha<\beta$ then $\alpha\oplus\gamma<\beta\oplus\gamma$, if $\alpha\le\beta$ then $\alpha\otimes\gamma\le\beta\otimes\gamma$, and if $\alpha<\beta$ and $\gamma>0$ then $\alpha\otimes\gamma<\beta\otimes\gamma$. Furthermore, $\alpha+\beta\le\alpha\oplus\beta$ and $\alpha\beta\le\alpha\otimes\beta$.

Recall that if $\alpha=\omega^{\alpha_1}+\cdots+\omega^{\alpha_k}$ is in CNF, then $\alpha\omega = \sup_{n\in\mathbb N} \alpha n = \omega^{\alpha_1+1}$, and $\alpha^\omega = \sup_{n\in\mathbb N} \alpha^n= \omega^{\alpha_1\omega}$. Then the following properties are readily checked:
\begin{itemize}
    \item $\underbrace{\alpha\oplus\cdots\oplus\alpha}_n = \alpha\otimes n,\qquad n\in\mathbb N$;
    \item $\sup_{n\in\mathbb N}\alpha\oplus n = \alpha+\omega$ (not $\alpha\oplus\omega$!);
    \item if $\alpha$ and $\beta$ are limit ordinals, then $\alpha\oplus\beta = \sup_{\alpha'<\alpha,\beta'<\beta} (\alpha'\oplus\beta')$;
    \item $\sup_{n\in\mathbb N}\alpha\otimes n = \alpha\omega$ (not $\alpha\otimes\omega$!);
    \item if both $\alpha<\omega^\gamma$ and $\beta<\omega^\gamma$ then $\alpha\oplus\beta<\omega^\gamma$;
    \item if both $\alpha<\omega^{\omega^\gamma}$ and $\beta<\omega^{\omega^\gamma}$ then $\alpha\otimes\beta<\omega^{\omega^\gamma}$.
\end{itemize}

Define the \emph{repeated natural product} by transfinite induction, by letting $\alpha^{[0]} = 1$, $\alpha^{[\beta+1]}=\alpha^{[\beta]}\otimes\alpha$, and $\alpha^{[\beta]}=\sup_{\gamma<\beta}\alpha^{[\gamma]}$ for limit $\beta$. It can be checked that
\begin{equation*}
    \alpha^{[\omega]} = \sup_{n\in\mathbb N} \alpha^{[n]} = \sup_{n\in\mathbb N} (\alpha\otimes\cdots\otimes\alpha) = \alpha^\omega.
\end{equation*}
In general, for limit $\beta$ we have $\alpha^{[\beta]} = \alpha^\beta$, as can be shown by ordinal induction on $\beta$. It can also be shown by ordinal induction on $\beta$ that $(\omega^{\omega^\alpha})^{[\beta]} = (\omega^{\omega^\alpha})^\beta$. (See also Altman \cite{altman}.)

\subsection{Well partial orders}\label{subsec_WPO}

Given a set $A$ partially ordered by $\preceq$, a \emph{bad sequence} is a sequence $a_1,a_2,a_3\ldots$ of elements of $A$ such that there exist no indices $i<j$ for which $a_i\preceq a_j$. Then $\preceq$ is said to be a \emph{well partial order} (WPO) if there exist no infinite bad sequences of elements of $A$. The \emph{ordinal type} of $A$, denoted $o(A)$, is the ordinal height of the root of the tree of bad sequences of $A$. It also equals the maximal order type of a linear order $\le$ extending $\preceq$ (Blass and Gurevich \cite{blass_gurevich}, see also de Jongh and Parikh \cite{dejongh_parikh}).

Given WPOs $A$ and $B$, their disjoint union $A\sqcup B$ can be well partially ordered by letting $x\preceq y$ if and only if $x,y\in A$ and $x\preceq_A y$ or $x,y\in B$ and $x\preceq_B y$. Then $o(A\sqcup B) = o(A)\oplus o(B)$ \cite{dejongh_parikh}. Also, their Cartesian product $A\times B$ can be well partially ordered by letting $(a,b)\preceq(a',b')$ if and only if $a\preceq_A a'$ and $b\preceq_B b'$. Then $o(A\times B) = o(A)\otimes o(B)$ \cite{dejongh_parikh}.
    
    \section{Proof of Theorem \ref{M halts}}
    
    We start by proving some properties of ordinal decreasing functions.
    \begin{Lemma}
    \label{perfect subsequence}
        Suppose $f:D\to\mathbb R$ is ordinal decreasing. Then for every infinite decreasing sequence $\{x_n\}$ in $D$ there is an infinite subsequence $\{x'_{n}\}$ for which $\{f(x'_{n})\}$ is nondecreasing.
    \end{Lemma}
    \begin{proof}
        By the infinite Ramsey's theorem \cite{ramsey}. Let an infinite decreasing sequence $\{x_n\}$ be given. Define an infinite complete graph in which there is a vertex for each $x_{i}$ and color each edge $\{x_{i},x_{j}\},i<j$ red if $f(x_{i})>f(x_{j})$ and green otherwise.
        
        An $f$-bad sequence corresponds to a monochromatic red complete subgraph. Hence, since $f(x)$ is ordinal decreasing in $D$, our graph cannot contain a monochromatic red infinite complete subgraph. Therefore, by the infinite Ramsey's theorem, our graph must contain a monochromatic green infinite subgraph, and thus the original sequence contains an infinite nondecreasing subsequence (comprised of all the vertices in the subgraph).
    \end{proof}
    \begin{Lemma}
    \label{closed to composition}
        Suppose $g(x)$ is ordinal decreasing in $D$ and satisfies $-g(x)<y$ for all $x\in D$, and suppose $f(x)$ is ordinal decreasing in $(-\infty,y)$. Then $f(-g(x))$ is ordinal decreasing in $D$.
    \end{Lemma}
    \begin{proof}
        Consider an infinite decreasing sequence $\{x_{n}\}$ in $D$. By Lemma \ref{perfect subsequence}, there exists an infinite nondecreasing subsequence of $\{g(x_{n})\}$. If this subsequence is not strictly increasing, there exists  $i<j$ such that $g(x_i)=g(x_j)$, and so is $f(-g(x_i))=f(-g(x_j))$. Otherwise we have a strictly decreasing subsequence of $\{-g(x_n)\}$, and since $f(x)$ is ordinal decreasing in $(-\infty,y)$, we can find $f(-g(x_j))\geq f(-g(x_i))$ with $i<j$, and we are done.
    \end{proof}
    \begin{Lemma}
        \label{closed to addition}
        Suppose $f(x)$ and $g(x)$ are ordinal decreasing in $D$. Then $f(x)+g(x)$ is ordinal decreasing in $D$.
    \end{Lemma}
    \begin{proof}
        Consider an infinite decreasing sequence $\{x_n\}$ in $D$. By Lemma \ref{perfect subsequence} on $f$, we can find an infinite subsequence $\{x'_n\}$ such that $\{f(x'_n)\}$ is nondecreasing. Since $g$ is ordinal decreasing in $D$, we can find in the subsequence $\{x'_n\}$ a pair of indices $i<j$ such that $g(x'_j) \geq g(x'_i)$. Hence, for these indices we have both $f(x'_j)\ge f(x'_i)$ and $g(x'_j)\ge g(x'_i)$, which implies $f(x'_j)+g(x'_j) \geq f(x'_i)+g(x'_i)$. Thus, $f(x)+g(x)$ is ordinal decreasing in $D$, as desired.
    \end{proof}
    (In this paper we only use Lemma \ref{closed to addition} for the special case $g(x)=-x$.)
    \begin{Lemma}
    \label{eventual increase}
        Let $\beta>0$, and suppose $f(x)$ is ordinal decreasing in $D=(y,y+\beta)$ and larger than $0$. Then there exists an $\varepsilon\in(0,\beta)$ such that for every $x\in(y,y+\varepsilon)$ we have $x-f(x)<y$.
    \end{Lemma}
    \begin{proof}
        Suppose for a contradiction that for every $\varepsilon\in(0,\beta)$ we have a counterexample $x\in(y,y+\varepsilon)$ with $ x-f(x) \geq y $. Then, we have an infinite decreasing sequence $\{x_n\}$ of such counterexamples with $\lim(x_n)=y$, but because $0<f(x_n)\leq x_n-y$ there exists an infinite subsequence $\{x'_n\}$ for which $f(x'_n)$ is also decreasing. This contradicts the assumption that $f(x)$ is ordinal decreasing in $D$.
    \end{proof}
    We are ready to prove our main result.
    \begin{proof}[Proof of Theorem \ref{M halts}]
        Consider the recursive algorithm
        \begin{equation}
        M(x) = 
            \begin{cases}
                f(x), & \text{if  $x < 0$}\\
                g_1(-M(x-g_2(-M(x-\cdots-g_k(-M(x-s(x)))\cdots)))) & \text{if  $x \geq 0$}
            \end{cases}
        \end{equation}
    where $s(x)$, $f(x)$ and $g_i(x)$ for all $i$ are all ordinal decreasing and larger than 0 for every $x$ in the appropriate ranges: $(-\infty, 0)$ for $f,g_i$, and $[0,\infty)$ for $s$.
    
    We claim that $M(x)$ halts and is ordinal decreasing for every $x$. We will prove this by real induction (Lemma \ref{real induction}).

    Let 
    \begin{equation*}
        S=\{x\mid \text{$M$ is defined and ordinal decreasing in $(-\infty,x]$} \}.
    \end{equation*}
Since for $x<0$ $M(x)$ is defined by $M(x)=f(x)$, we have $(-\infty ,0)\subseteq S$. Hence, S satisfies property (R1).

Next, suppose that $(-\infty,y) \subseteq S$. Then note that $M(y)$ is defined, since for every $i$ $M(y-g_i(-M(y-\cdots-g_k(-M(y-s(y)))\cdots)))$ is defined by induction, since functions $g_i$ and $s$ have output larger than $0$. Hence, $M(x)$ is ordinal decreasing in $(-\infty,y]$. Hence $y \in S$ as well, so $S$ satisfies property (R2).

Finally suppose $y \in S$. We will show that $(y,y+\varepsilon)\subseteq S$ for some $\varepsilon>0$, meaning $S$ satisfies property (R3). If $y<0$ then the existence of such an $\varepsilon$ follows immediately, so suppose $y\ge 0$.

In order to show the existence of the desired $\varepsilon$, we will show by induction on $i=k,\ldots, 1$ that $M_i(x)=g_i(-M(x-\cdots-g_k(-M(x-s(x)))\cdots))$ is defined and ordinal decreasing in $(-\infty,y+\varepsilon_i)$ for some $\varepsilon_i>0$.
Let us start with the base case $i=k$. In this case $M_k(x)= g_k(-M(x-s(x)))$. By Lemma \ref{closed to addition}, $-x+s(x)$ is an ordinal decreasing function in $[0,\infty)$. By Lemma \ref{eventual increase} on $s(x)$, there is an $\varepsilon_k$ such that $x-s(x)<y$ for every $y<x<y+\varepsilon_k$. Hence, by assumption and Lemma \ref{closed to composition}, $M(-(-x+s(x)))=M(x-s(x))$ is ordinal decreasing and defined in $(-\infty,y+\varepsilon_k)$ and so is $g_k(-M(x-s(x)))$, as desired.
For the induction step, suppose $M_i(x)$ is ordinal decreasing and defined in $(-\infty,y+\varepsilon_i)$. 
By Lemma \ref{closed to addition}, $-x+M_i(x)$ is ordinal decreasing and defined in $(-\infty,y+\varepsilon_i)$. Furthermore, by Lemma \ref{eventual increase} there exists some $\varepsilon_{i-1}$ such that $x-M_i(x)<y$ for every $y<x<y+\varepsilon_{i-1}$. Hence, by assumption and Lemma \ref{closed to composition} $M(x-M_i(x))$ is defined and ordinal decreasing in $(-\infty,y+\varepsilon_{i-1})$. Hence, so is $M_{i-1}=g_{i-1}(-M(x-M_i(x)))$. Therefore, $S$ satisfies property (R3) as well. Thus, by Lemma \ref{real induction}, we have $S=\mathbb{R}$.
    \end{proof}
    
\section{Proof of Theorem \ref{thm_o_bound}}

\begin{figure}
    \centering
            \includegraphics{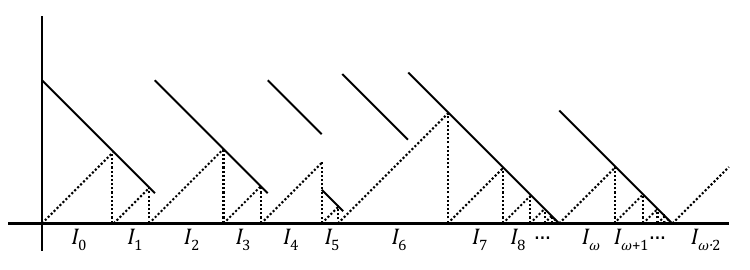}
            \caption{An ordinal decreasing function (solid line) induces a partition of the $x$-axis into a transfinite number of intervals (dotted lines).}
            \label{fig_make_intervals}
\end{figure}

Let $D=[x_{\mathrm{lo}},x_{\mathrm{hi}})\subset\mathbb R$ be an interval, where $x_{\mathrm{lo}}\neq -\infty$ but $x_{\mathrm{hi}}$ could be $\infty$. Let $f:D\to\mathbb R$ be an ordinal decreasing function that is positive for all $x\in D$. Recall that Lemma \ref{eventual increase} implies that for every $p\in D$ there exists $y\in(p,x_{\mathrm{hi}}]$ such that for all $x\in(p,y)$ we have $x-p<f(x)$. Hence, the function $f$ induces a partition of $D$ into maximal intervals as follows. Define the endpoints $p_\alpha$ by
\begin{align*}
    p_0&=x_{\mathrm{lo}};\\
    p_{\alpha+1}&=\max{\{y\le x_{\mathrm{hi}}: x-p_\alpha<f(x) \text{ for all $x<y$}\}};\\
    p_\alpha &= \lim_{\beta<\alpha} p_\beta,\qquad\text{$\alpha$ limit.}
\end{align*}
This definition stops when we reach $p_{\alpha^*}=x_{\mathrm{hi}}$ for some ordinal $\alpha^*$. Then define the intervals $I_\alpha = [p_\alpha, p_{\alpha+1})$ for ordinals $\alpha<\alpha^*$. These intervals form a partition of $D$. Figure \ref{fig_make_intervals} shows how the intervals $I_\alpha$ can be computed graphically: Starting at $x_{\mathrm{lo}}$ on the $x$-axis, we move up-right in a straight line with slope $1$, until we encounter the graph of $f$ or pass above the graph. At that point, we descend to the $x$-axis, mark a new endpoint $p_\alpha$, and start this process again.

\begin{Lemma}\label{lem_num_intervals}
    The ordinal number of intervals $I_\alpha$ into which $D$ is partitioned is at most $\omega\cdot(o(f|_D)+1)$.
\end{Lemma}

\begin{proof}
Recall that $f$ is positive for all $x\in D$. Call $x\in D$ a \emph{near-root} of $f$ if there exists an infinite increasing sequence $y_1,y_2,y_3,\ldots\in D$ such that $\lim_{n\to\infty} y_n = x$ and $\lim_{n\to\infty} f(y_n) = 0$. Let $L_f\subset \mathbb R$ be the set of near-roots of $f$. The set $L_f$ is well-ordered in $\mathbb R$, since from an infinite decreasing sequence of near-roots $x_1>x_2>x_3>\cdots$ we could construct an infinite $f$-bad sequence $z_1>z_2>z_3>\cdots$ by letting $z_1=x_1$, and for each $i=2,3,\ldots$ choosing $z_i\in(x_i,x_{i-1})$ such that $f(z_i)<f(z_{i-1})$. Hence, the set $L_f$ has an ordinal type $o(L_f)$.

Furthermore, the set $L_f$ is closed: If $x_1<x_2<x_3<\cdots$ are near-roots of $f$ that converge to $x$, then $x$ is a near-root of $f$ as well. Hence, the elements of $L_f$ can be classified into \emph{limit} and \emph{non-limit} near-roots. 

Let $\alpha\mapsto x_\alpha$, for $\alpha<o(L_f)$, be the order-preserving enumeration function of $L_f$. Note that $\alpha$ is a limit ordinal if and only if $x_\alpha$ is a limit near-root of $f$.

We will now show that $o(f|_D)\ge o(L_f)$.
For this, we prove by ordinal induction on $\alpha$ that the following is true for every $\alpha<o(L_f)$: For every $\beta<\alpha$ and every $\varepsilon>0$ there exists $y$ such that:
\begin{itemize}
    \item $x_\alpha-\varepsilon<y<x_\alpha$,
    \item $f(y)<\varepsilon$, and
    \item $o_f(y)\ge \beta$.
\end{itemize}
Indeed, given $\alpha$, given $\beta<\alpha$, and given $\varepsilon>0$, we can choose $y$ close enough to $x_\alpha$ such that $y\ge x_\beta$, $y>x_\alpha-\varepsilon$, and $f(y)<\varepsilon$. Then, applying the induction assumption on $\beta$, for every $\gamma<\beta$ we can find $z<x_\beta$ such that $f(z)<f(y)$ and $o_f(z)\ge \gamma$. It follows that $o_f(y)\ge \beta$ as desired.

Hence, we indeed have $o(f|_D)\ge o(L_f)$, as claimed.

Next, we determine the relation between the intervals $I_\alpha$ and the elements of $L_f$:

\begin{Observation}\label{obs_w-many}
    Let $z\in D$. Then there exists an infinite sequence of $\omega$ consecutive intervals $I_{\alpha}, I_{\alpha+1},\allowbreak I_{\alpha+2},\ldots$ that converge to $z$ if and only if $z$ is a non-limit near-root of $f$.
\end{Observation}

\begin{proof}
Suppose first that $z$ is not a near-root of $f$. Then there exists an $\varepsilon\in(0,x_{\mathrm{hi}}-z)$ such that $f(x)\ge\varepsilon$ for all $x\in [z-\varepsilon,z+\varepsilon]$ (where the part $x\in[z-\varepsilon, z]$ follows from the fact that $z$ is not a near-root, and the part $x\in[z, z+\varepsilon]$ follows from the fact that $f$ is ordinal decreasing). Therefore, an interval $I_\alpha$ whose left endpoint is in $(z-\varepsilon,z)$ must contain $z$ in its interior. Hence, there are not $\omega$-many intervals converging to $z$.

Now suppose $z$ is a near-root of $f$. Then no interval $I_\alpha$ whose left endpoint is left of $z$ can contain $z$ in its interior. If $z$ is a non-limit near-root of $f$, then there exists an $\varepsilon>0$ such that $(z-\varepsilon,z)$ contains no near-roots of $f$. Hence, for every $\varepsilon'<\varepsilon$, the interval $(z-\varepsilon,z-\varepsilon')$ contains only finitely many intervals $I_\alpha$. And therefore, there exist $\omega$-many consecutive intervals $I_\alpha$ converging to $z$. If, on the other hand, the near-root $z$ is itself a limit of near-roots of $f$, then some left-neighborhood of $z$ contains at least $\omega^2$-many intervals $I_\alpha$.
\end{proof}

Observation \ref{obs_w-many} implies that there is a one-to-one correspondence between non-limit elements of $L_f$ and sequences of $\omega$-many consecutive intervals $I_\alpha$, except for a possible final sequence of at most $\omega$-many intervals after the last element of $L_f$.
Lemma \ref{lem_num_intervals} follows.
\end{proof}

\begin{Lemma}\label{lem_J1J2}
    Let $J\subseteq\mathbb R$ be an interval, and let $J_1$, $J_2$ be a partition of $J$ into two intervals, with $J_1$ left of $J_2$. Then $o(f|_J)\le o(f|_{J_1})+o(f|_{J_2})$.
\end{Lemma}

\begin{proof}
    Every $f$-bad sequence in $J$ can be partitioned into an $f$-bad sequence in $J_2$ followed by an $f$-bad sequence in $J_1$ (though the converse is not necessarily true). Hence, the tree of $f$-bad sequences $T_{(f|_J)}$ is a subtree of the tree formed by attaching a copy of $T_{(f|_{J_1})}$ to each leaf of $T_{(f|_{J_2})}$. The ordinal type of this latter tree is $o(f|_{J_1})+o(f|_{J_2})$, so the claim follows.
\end{proof}

\begin{Lemma}\label{lem_fx_x}
    Let $f$ be ordinal decreasing, and let $g(x) = f(x)-x$ (which is ordinal decreasing by Lemma \ref{closed to addition}). Then $o(g)\le o(f)$.
\end{Lemma}

\begin{proof}
    Every $g$-bad sequence is also $f$-bad, hence $T_g\subseteq T_f$.
\end{proof}

\begin{Lemma}\label{lem_o_composition}
    Suppose $g(x)$ is ordinal decreasing in $(-\infty,y)$ and satisfies $-g(x)<y'$ for all $x<y$, and suppose $f(x)$ is ordinal decreasing in $(-\infty,y')$. Let $h(x) = f(-g(x))$ (which is ordinal decreasing in $(-\infty,y)$ by Lemma \ref{closed to composition}). Then $o(h|_y)\le o(g|_y)\otimes o(f|_{y'})$.
\end{Lemma}

\begin{proof}
    Let $A=o(g|_y)\times o(f|_{y'})$ be WPO by the standard product order mentioned in Section \ref{subsec_WPO}. Given $x\le y$, let $E(x) = (o_g(x), o_f(-g(x)))\in A$. Then $E$ satisfies the following property:
    \begin{Lemma}
    If $x>x'$ and $E(x)\preceq E(x')$ then $h(x)\le h(x')$. (Hence, $E$ is analogous to what Rathjen and Weiermann \cite{RATHJEN199349} call a \emph{quasi-embedding}.) 
    \end{Lemma}
    \begin{proof}
        We have $x>x'$ and $o_g(x)\le o_g(x')$. Hence, $g(x)\le g(x')$ (because $x>x'$ and $g(x)>g(x')$ would imply $o_g(x)>o_g(x')$). If $g(x)=g(x')$ then $h(x)=h(x')$ and we are done. Otherwise, $g(x)<g(x')$, so $-g(x)>-g(x')$. We also have $o_f(-g(x))\le o_f(-g(x'))$. Hence, $h(x)=f(-g(x))\le f(-g(x'))=h(x')$, as desired.
    \end{proof}
    Hence, if $x_1> x_2 > \cdots > x_n$ is an $h$-bad sequence then $E(x_1), E(x_2), \ldots, E(x_n)$ is a bad sequence in $A$. Therefore, $o(h|_y)\le o(A) = o(g|_y)\otimes o(f|_{y'})$. This completes the proof of Lemma \ref{lem_o_composition}.
\end{proof}

The following lemma is not actually used in this paper, but it might be of independent interest:

\begin{Lemma}
    Suppose $f$ and $g$ are ordinal decreasing, and let $h(x)=f(x)+g(x)$ (which is ordinal decreasing by Lemma \ref{closed to addition}). Then $o(h)\le o(f)\otimes o(g)$.
\end{Lemma}

\begin{proof}
    The claim follows by considering the quasi-embedding $E(x) = (o_f(x), o_g(x))$.
\end{proof}

\subsection{The case $k=1$}

When $k=1$ the algorithm is
\begin{equation*}
    M(x) = \begin{cases}f(x),& x<0;\\g(-M(x-s(x))),& x\ge 0.\end{cases}
\end{equation*}

Consider the partition of $[0,\infty)$ into intervals induced by $s$. Namely, let
\begin{align*}
    p_0 &= 0;\\
    p_{\alpha+1} &= \max{\{ y : x- p_\alpha < s(x) \text{ for all $x<y$}\}};\\
    p_\alpha &= \lim_{\beta<\alpha} p_\beta, \qquad \text{for $\alpha$ limit}.
\end{align*}
Then define the intervals $I_{-1} = (-\infty, 0)$ and $I_\alpha = [p_\alpha, p_{\alpha+1})$ for ordinals $\alpha$.

Denote $\tau_\alpha = o(M|_{p_\alpha})$. We will compute $\tau_\alpha$ by ordinal induction. The base case is $\alpha=0$, for which $p_0=0$, and thus $\tau_0 = o(M|_{I_{-1}}) = o(f)$.

If $x\in I_{\alpha} = [p_\alpha, p_{\alpha+1})$ then by the definition of $p_{\alpha+1}$ we have $x-s(x) < p_\alpha$. Therefore, by Lemma \ref{lem_J1J2}, Lemma \ref{lem_fx_x}, and two applications of Lemma \ref{lem_o_composition},
\begin{align*}
    \tau_{\alpha+1} = o(M|_{ p_{\alpha+1}}) &\le o(M|_{p_\alpha}) + o(M|_{I_\alpha}) \\
    &\le \tau_\alpha+o(g)\otimes o(M|_{p_\alpha})\otimes o(s(x)-x)\\
    &\le \tau_\alpha + o(g)\otimes\tau_\alpha \otimes o(s).
\end{align*}

Let $\gamma$ be large enough such that $\max{\{o(f),o(g),o(s)\}} < \omega^{\omega^\gamma}$. Then $\tau_0 = o(f)\le\omega^{\omega^\gamma}$, and it follows by ordinal induction on $\alpha$ that $\tau_\alpha\le \omega^{\omega^\gamma(1+\alpha)}$. By Lemma \ref{lem_num_intervals}, we conclude that $o(M)\le \omega^{\omega^{\gamma+1}(o(s)+1)}$, as desired.

\subsection{The case $k=2$}

When $k=2$ the algorithm is
\begin{equation*}
    M(x) = \begin{cases}f(x),& x<0;\\g_1(-M(x-g_2(-M(x-s(x))))),& x\ge 0.\end{cases}
\end{equation*}

Denote $M_2(x) = g_2(-M(x-s(x)))$, and $M(x) = M_1(x) = g_1(-M(x-M_2(x)))$. Define the points $p_\alpha$ and the intervals $I_\alpha$ as above, based on the function $s$.

Partition each interval $I_\alpha$ into subintervals $I_{\alpha,\beta}$ based on the function $M_2$, as follows. Define points $p_{\alpha,\beta}$ by
\begin{align*}
    p_{\alpha,0} &= p_\alpha;\\
    p_{\alpha,\beta+1} &= \max{\{ y\le p_{\alpha+1} : x- p_{\alpha,\beta}<M_2(x) \text{ for all $x<y$}\}};\\
    p_{\alpha,\beta} &= \lim_{\beta'<\beta} p_{\alpha,\beta'}, \qquad \text{for $\beta$ limit}.
\end{align*}
Then define the subintervals $I_{\alpha,\beta}=[p_{\alpha,\beta}, p_{\alpha,\beta+1})$.

Denote $\tau_\alpha = o(M|_{p_\alpha})$ and $\tau_{\alpha,\beta} = o(M|_{p_{\alpha,\beta}})$. Also denote $\sigma_\alpha=o(M_2|_{I_\alpha})$.

\begin{Lemma}\label{lem_sigma_alpha}
    We have $\sigma_{\alpha}\le \tau_\alpha\otimes o(s)\otimes o(g_2)$.
\end{Lemma}

\begin{proof}
    For $x\in I_\alpha$ we have $x-s(x)<p_\alpha$ (by the definition of $p_{\alpha+1}$). Hence,
    \begin{equation*}
        \sigma_\alpha = o(M_2|_{I_\alpha})\le o(g_2)\otimes o(M|_{p_\alpha})\otimes o(s(x)-x)\le o(g_2)\otimes\tau_\alpha\otimes o(s),
    \end{equation*}
     by Lemmas \ref{lem_fx_x} and \ref{lem_o_composition}.
\end{proof}

\begin{Lemma}\label{lem_tau_alpha_beta}
    We have $\tau_{\alpha,0} = \tau_\alpha$ and $\tau_{\alpha,\beta+1} \le \tau_{\alpha,\beta}+\tau_{\alpha,\beta}\otimes\tau_\alpha\otimes o(s)\otimes o(g_2)\otimes o(g_1)$.
\end{Lemma}

\begin{proof}
    The first claim follows by definition. For the second claim, Lemma \ref{lem_J1J2} yields
    \begin{equation*}
        \tau_{\alpha,\beta+1}=o(M|_{p_{\alpha,\beta+1}}) \le o(M|_{p_{\alpha,\beta}})+o(M|_{I_{\alpha,\beta}})=\tau_{\alpha,\beta}+o(M|_{I_{\alpha,\beta}}).
    \end{equation*}
    Let us bound $o(M|_{I_{\alpha,\beta}})$. If $x\in I_{\alpha,\beta}$ then $x-M_2(x)<p_{\alpha,\beta}$ (by the definition of $p_{\alpha,\beta+1}$). Hence, by Lemmas \ref{lem_fx_x} and \ref{lem_o_composition},
    \begin{align*}
        o(M|_{I_{\alpha,\beta}})&\le o(g_1)\otimes o(M|_{p_{\alpha,\beta}})\otimes o((M_2(x)-x)|_{I_{\alpha,\beta}})\\
        &\le o(g_1)\otimes \tau_{\alpha,\beta}\otimes o(M_2|_{I_{\alpha,\beta}})\\
        &\le o(g_1)\otimes\tau_{\alpha,\beta}\otimes o(M_2|_{I_\alpha})=o(g_1)\otimes\tau_{\alpha,\beta}\otimes\sigma_\alpha,
    \end{align*}
    since $I_{\alpha,\beta}\subseteq I_\alpha$. The bound for $\tau_{\alpha,\beta+1}$ now follows by Lemma \ref{lem_sigma_alpha}.
\end{proof}

From Lemma \ref{lem_tau_alpha_beta} it follows, by transfinite induction on $\beta$, that $\tau_{\alpha,\beta}$ is bounded by the repeated natural product $\tau_{\alpha,\beta}\le \tau_\alpha \otimes (\tau_\alpha\otimes o(s)\otimes o(g_2)\otimes o(g_1))^{[\beta]}$ (see Section \ref{subsec_natural} above). Let us take $\beta$ to be the ordinal number of subintervals into which interval $I_\alpha$ is partitioned, so $\tau_{\alpha+1}=\tau_{\alpha,\beta}$. By Lemmas \ref{lem_num_intervals} and \ref{lem_sigma_alpha} we have $\beta\le \omega\cdot(\sigma_\alpha+1)\le \omega\cdot(\tau_\alpha\otimes o(s)\otimes o(g_2)+1)$. Hence:
\begin{Corollary}\label{cor_tau_ap1}
    We have $\tau_{\alpha+1} \le (\tau_\alpha\otimes o(s)\otimes o(g_2)\otimes o(g_1))^{[\omega\cdot(\tau_\alpha\otimes o(s)\otimes o(g_2)+2)]}$.
\end{Corollary}

\begin{Lemma}\label{lem_tau_omega_beta}
    Let $\gamma$ be smallest such that $\max{\{o(f),o(g_1), o(g_2), o(s)\}} < \varepsilon_\gamma$. Then $\tau_{\omega\beta}\le\varepsilon_{\gamma+\beta}$ for all $\beta$.
\end{Lemma}

\begin{proof}
Denote $\mathrm{tower}(\delta,k)=\delta^{\delta^{\cdots^\delta}}$ with $k$ instances of $\delta$. It can be checked that, if $\delta\ge \omega$ and $k\le \ell$, then $\mathrm{tower}(\delta,k)^{\mathrm{tower}(\delta,\ell)}=\mathrm{tower}(\delta,\ell+1)$.

    We proceed by ordinal induction on $\beta$. If $\beta=0$ then $\tau_0=o(f)<\varepsilon_\gamma$, as desired.

    Now suppose by ordinal induction that $\tau_{\omega\beta}\le\varepsilon_{\gamma+\beta}$. Since $\tau_{\omega\beta},o(s),o(g_1),o(g_2)$ are all bounded by $\varepsilon_{\gamma+\beta}$, Corollary \ref{cor_tau_ap1} implies
    \begin{equation*}
\tau_{\omega\beta+1}\le(\varepsilon_{\gamma+\beta}^{[4]})^{[\omega\cdot(\varepsilon_{\gamma+\beta}^{[3]}+2)]} = \varepsilon_{\gamma+\beta}^{\omega\cdot(\varepsilon_{\gamma+\beta}^{3}+2)}\le\varepsilon_{\gamma+\beta}^{\varepsilon_{\gamma+\beta}^{\varepsilon_{\gamma+\beta}}}=\mathrm{tower}(\varepsilon_{\gamma+\beta},3).
    \end{equation*}
    Another application of Corollary \ref{cor_tau_ap1} yields
    \begin{equation*}
        \tau_{\omega\beta+2}\le(\mathrm{tower}(\varepsilon_{\gamma+\beta},3)^{[4]})^{[\omega\cdot(\mathrm{tower}(\varepsilon_{\gamma+\beta},3)^{[3]}+2)]}\le\mathrm{tower}(\mathrm{tower}(\varepsilon_{\gamma+\beta},3),3)=\mathrm{tower}(\varepsilon_{\gamma+\beta},5).
    \end{equation*}
    In general, by induction on $k$, we obtain $\tau_{\omega\beta+k}\le\mathrm{tower}(\varepsilon_{\gamma+\beta},2k+1)$. Hence,
    \begin{equation*}
        \tau_{\omega(\beta+1)}=
    \tau_{\omega\beta+\omega}\le\sup_{k\in\N}\mathrm{tower}(\varepsilon_{\gamma+\beta},k)=\varepsilon_{\gamma+\beta+1},
    \end{equation*}
    as desired. The case where $\beta$ is a limit ordinal follows by continuity.
\end{proof}

By Lemma \ref{lem_num_intervals}, the ordinal number of intervals $I_\alpha$ is at most $\omega\cdot(o(s)+1)$. Hence, Lemma \ref{lem_tau_omega_beta} implies that $o(M) \le \varepsilon_{\gamma + o(s)+1}$, as claimed.

\subsection{The general case}

The algorithm for general $k$ for $x\ge 0$ is
\begin{align*}
M(x) = M_1(x) &= g_1(-M(x-M_2(x))), \quad\text{where}\\
M_2(x) &= g_2(-M(x-M_3(x))), \quad\text{where}\\
&\vdots\\
M_{k-1}(x) &= g_{k-1}(-M(x-M_k(x))), \quad\text{where}\\
M_k(x) &= g_k(-M(x-s(x))).
\end{align*}
Define the endpoints $p_{\alpha_1,\ldots,\alpha_i}$ for $i=1,\ldots,k$ by induction on $i$, and for each $i$ by ordinal induction on $\alpha_i$, as follows. For $i=1$ we let
\begin{align*}
    p_0 &= 0;\\
    p_{\alpha_1+1} &= \max{\{y: x-p_{\alpha_1}<s(x) \text{ for all $x<y$}\}};\\
    p_{\alpha_1} &= \lim_{\beta<\alpha_1} p_\beta,\qquad \text{for $\alpha_1$ limit}.
\end{align*}
For $2\le i\le k$ we let
\begin{align*}
    p_{\alpha_1,\ldots,\alpha_{i-1},0}&=p_{\alpha_1,\ldots,\alpha_{i-1}};\\
    p_{\alpha_1,\ldots,\alpha_{i-1},\alpha_i+1} &= \max{\{y\le p_{\alpha_1,\ldots,\alpha_{i-1}+1}: x-p_{\alpha_1,\ldots,\alpha_i}<M_{k-i+2}(x) \text{ for all $x<y$}\}};\\
    p_{\alpha_1,\ldots,\alpha_i}&=\lim_{\beta<\alpha_i} p_{\alpha_1,\ldots,\alpha_{i-1},\beta},\qquad\text{for $\alpha_i$ limit}.
\end{align*}

We define a $k$-ary hierarchy of intervals: For each $1\le i\le k$, we let $I_{\alpha_1,\ldots,\alpha_i} = [p_{\alpha_1,\ldots,\alpha_i},p_{\alpha_1,\ldots,\alpha_{i-1},\alpha_i+1})$. We call this a \emph{level-$i$} interval. Hence, each level-$1$ interval is partitioned into level-$2$ intervals, each of which is partitioned into level-$3$ intervals, and so on, up to level-$k$ intervals.

For $1\le i\le k$, define the ordinals $\tau_{\alpha_1,\ldots,\alpha_i}=o(M|_{p_{\alpha_1,\ldots,\alpha_i}})$, and the ordinals $\sigma_{\alpha_1,\ldots,\alpha_i}=o(M_{k-i+1}|_{I_{\alpha_1,\ldots,\alpha_i}})$. By Lemma \ref{lem_num_intervals}, for each $1\le i\le k-1$, the level-$i$ interval $I_{\alpha_1,\ldots,\alpha_i}$ is partitioned into at most
\begin{equation}\label{eq_num_subintervals}
    \omega\cdot(\sigma_{\alpha_1,\ldots,\alpha_i}+1)
\end{equation}
level-$(i+1)$ subintervals.

\begin{Lemma}\label{lem_sigma}
    For every $1\le i\le k$ we have
    \begin{equation*}
    \sigma_{\alpha_1,\ldots,\alpha_i}\le o(s)\otimes o(g_k)\otimes\cdots\otimes o(g_{k-i+1})\otimes \tau_{\alpha_1}\otimes\cdots\otimes \tau_{\alpha_1,\ldots,\alpha_i}.
    \end{equation*}
\end{Lemma}

\begin{proof}
    By induction on $i$. First, let $i=1$. Denote $I=I_{\alpha_1}$. If $x\in I$, then the expression $x-s(x)$ (which is part of the definition of $M_k$) is smaller than $p_{\alpha_1}$ (by the definition of $p_{\alpha_1+1}$). Hence, by Lemmas \ref{lem_fx_x} and \ref{lem_o_composition},
    \begin{align*}
        \sigma_{\alpha_1} = o(M_k|_I)&\le o(g_k)\otimes o(M|_{p_{\alpha_1}})\otimes o(s(x)-x)\\
        &\le o(g_k)\otimes \tau_{\alpha_1}\otimes o(s),
    \end{align*}
    as claimed.

    Now, suppose $i\ge 2$, and assume the bound on $\sigma_{\alpha_1,\ldots,\alpha_{i-1}}$ by induction. Denote $I=I_{\alpha_1,\ldots,\alpha_i}$. If $x\in I$, then the expression $x-M_{k-i+2}(x)$ (which is in the definition of $M_{k-i+1}$) is smaller than $p_{\alpha_1,\ldots,\alpha_i}$. Hence,
    \begin{align*}
        \sigma_{\alpha_1,\ldots,\alpha_i}=o(M_{k-i+1}|_{I})&\le o(g_{k-i+1})\otimes o(M|_{p_{\alpha_1,\ldots,\alpha_i}})\otimes o((M_{k-i+2}(x)-x)|_{I})\\
        &\le o(g_{k-i+1})\otimes\tau_{\alpha_1,\ldots,\alpha_i}\otimes o(M_{k-i+2}|_I).
    \end{align*}
    Since $I\subseteq I_{\alpha_1,\ldots,\alpha_{i-1}}$, we have $o(M_{k-i+2}|_I)\le o(M_{k-i+2}|_{I_{\alpha_1,\ldots,\alpha_{i-1}}}) = \sigma_{\alpha_1,\ldots,\alpha_{i-1}}$. The claimed bound on $\sigma_{\alpha_1,\ldots,\alpha_i}$ follows.
\end{proof}

\begin{Lemma}
    We have
    \begin{align*}
        \tau_{\alpha_1,\ldots,\alpha_{k-1},0}&=\tau_{\alpha_1,\ldots,\alpha_{k-1}},\\
        \tau_{\alpha_1,\ldots,\alpha_{k-1},\alpha_k+1} &\le \tau_{\alpha_1,\ldots,\alpha_k} + \tau_{\alpha_1}\otimes\cdots\otimes\tau_{\alpha_1,\ldots,\alpha_k}\otimes o(s)\otimes o(g_1)\otimes\cdots\otimes o(g_k).
    \end{align*}
\end{Lemma}

\begin{proof}
    The first equation is immediate from the definition of $p_{\alpha_1,\ldots,\alpha_{i-1},0}$. For the second equation, Lemma \ref{lem_J1J2} yields $\tau_{\alpha_1,\ldots,\alpha_{k-1},\alpha_k+1}\le\tau_{\alpha_1,\ldots,\alpha_k}+\sigma_{\alpha_1,\ldots,\alpha_k}$ (since $M_1=M$), on which we apply Lemma \ref{lem_sigma}.
\end{proof}

\begin{Corollary}\label{cor_tau_k}
    The ordinal $\tau_{\alpha_1,\ldots,\alpha_k}$ is bounded by the repeated natural product
    \begin{equation*}
        \tau_{\alpha_1,\ldots,\alpha_k}\le \tau_{\alpha_1,\ldots,\alpha_{k-1}}\otimes\bigl(\tau_{\alpha_1}\otimes\cdots\otimes\tau_{\alpha_1,\ldots,\alpha_{k-1}}\otimes o(s)\otimes o(g_1)\otimes\cdots\otimes o(g_k)\bigr)^{[\alpha_k]}.
    \end{equation*}
\end{Corollary}

\begin{proof}
    By transfinite induction on $\alpha_k$.
\end{proof}

Denote $\delta=o(f)\otimes o(s)\otimes o(g_1)\otimes\cdots\otimes o(g_k)$. Hence, by Lemma \ref{lem_sigma},
\begin{equation}\label{eq_sigma_delta}
\sigma_{\alpha_1,\ldots,\alpha_i} \le \delta\otimes\tau_{\alpha_1,\ldots,\alpha_i}^{[i]}.
\end{equation}
Let us take $\alpha_k$ to be the ordinal number of subintervals into which the interval $I_{\alpha_1,\ldots,\alpha_{k-1}}$ is partitioned, so $\tau_{\alpha_1,\ldots,\alpha_{k-2},\alpha_{k-1}+1}=\tau_{\alpha_1,\ldots,\alpha_k}$. Equation (\ref{eq_num_subintervals}) bounds the value of $\alpha_k$. Together with Corollary \ref{cor_tau_k} and equation (\ref{eq_sigma_delta}), we get
\begin{equation}\label{eq_last_tau_plus_1}
    \tau_{\alpha_1,\ldots,\alpha_{k-1}+1}\le (\delta\otimes\tau_{\alpha_1,\ldots,\alpha_{k-1}}^{[k-1]})^{[\omega\cdot(\delta\otimes\tau_{\alpha_1,\ldots,\alpha_{k-1}}^{[k-1]}+2)]}.
\end{equation}

\begin{Lemma}\label{lem_varphi_i}
    Let $2\le i\le k-1$, and let $\alpha_1,\ldots, \alpha_{k-i}$ be given. Let $\rho=\rho(\alpha_1,\ldots,\alpha_{k-i})$ be sufficiently large such that
    \begin{equation*}
        \max{\bigl\{\delta,\tau_{\alpha_1,\ldots,\alpha_{k-i}}\bigr\}}<\varphi_{i-1}(\rho).
    \end{equation*}
    Then
    \begin{equation*}
        \tau_{\alpha_1,\ldots,\alpha_{k-i},\omega\beta}\le\varphi_{i-1}(\rho+\beta)\qquad\text{for every $\beta$.}
    \end{equation*}
\end{Lemma}

\begin{proof}
    By induction on $i$. First let $i=2$. Given $\alpha_1,\ldots,\alpha_{k-2}$, let $\rho$ be large enough so that $\varepsilon_\rho>\delta$ and $\varepsilon_\rho>\tau_{\alpha_1,\ldots,\alpha_{k-2}}$. We proceed by ordinal induction on $\beta$. If $\beta=0$ then the claim follows by the choice of $\rho$, since $\tau_{\alpha_1,\ldots,\alpha_{k-2},0}=\tau_{\alpha_1,\ldots,\alpha_{k-2}}$. Now suppose by induction on $\beta$ that $\tau_{\alpha_1,\ldots,\alpha_{k-2},\omega\beta}\le\varepsilon_{\rho+\beta}$. Equation (\ref{eq_last_tau_plus_1}) yields
    \begin{align*}
        \tau_{\alpha_1,\ldots,\alpha_{k-2},\omega\beta+1}&\le(\delta\otimes\tau_{\alpha_1,\ldots,\alpha_{k-2},\omega\beta}^{[k-1]})^{[\omega\cdot(\delta\otimes\tau_{\alpha_1,\ldots,\alpha_{k-2},\omega\beta}^{[k-1]}+2)]}\\&\le(\varepsilon_{\rho+\beta}^{[k]})^{[\omega\cdot(\varepsilon_{\rho+\beta}^{[k]}+2)]}\le\varepsilon_{\rho+\beta}^{\varepsilon_{\rho+\beta}^{\varepsilon_{\rho+\beta}}}=\mathrm{tower}(\varepsilon_{\rho+\beta},3).
    \end{align*}
    A second application of equation (\ref{eq_last_tau_plus_1}) yields $\tau_{\alpha_1,\ldots,\alpha_{k-2},\omega\beta+2}\le\mathrm{tower}(\mathrm{tower}(\varepsilon_{\rho+\beta},3),3)=\mathrm{tower}(\varepsilon_{\rho+\beta},5)$. In general, by induction on $j$ we get $\tau_{\alpha_1,\ldots,\alpha_{k-2},\omega\beta+j}\le\mathrm{tower}(\varepsilon_{\rho+\beta},2j+1)$. Hence,
    \begin{equation*}
        \tau_{\alpha_1,\ldots,\alpha_{k-2},\omega(\beta+1)}=\tau_{\alpha_1,\ldots,\alpha_{k-2},\omega\beta+\omega}=\sup_{j\in\N}\mathrm{tower}(\varepsilon_{\rho+\beta},j)=\varepsilon_{\rho+\beta+1},
    \end{equation*}
    as desired. The case where $\beta$ is a limit ordinal follows by continuity.    
    
    Now let $i\ge 3$, and suppose the claim is true for $i-1$. Given $\alpha_1,\ldots,\alpha_{k-i}$, let $\rho$ be large enough so that $\varphi_{i-1}(\rho)>\delta$ and $\varphi_{i-1}(\rho)>\tau_{\alpha_1,\ldots,\alpha_{k-i}}$. We proceed by ordinal induction on $\beta$. If $\beta=0$ the claim follows by the choice of $\rho$. Now suppose by induction on $\beta$ that $\tau_{\alpha_1,\ldots,\alpha_{k-i},\omega\beta}\le\varphi_{i-1}(\rho+\beta)$. By (\ref{eq_num_subintervals}) and (\ref{eq_sigma_delta}), the interval $I_{\alpha_1,\ldots,\alpha_{k-i},\omega\beta}$ is partitioned into at most $\omega\gamma$ subintervals for $\gamma=\delta\otimes\tau_{\alpha_1,\ldots,\alpha_{k-i},\omega\beta}^{[k-i+1]}+1$. Hence,
    \begin{equation}\label{eq_ends_in_omegagamma}
        \tau_{\alpha_1,\ldots,\alpha_{k-i},\omega\beta+1}\le\tau_{\alpha_1,\ldots,\alpha_{k-i},\omega\beta,\omega\gamma}.
    \end{equation}
    Let us bound the right-hand side of (\ref{eq_ends_in_omegagamma}) using the induction assumption on $i-1$. We first need to choose $\rho'$ large enough so that $\varphi_{i-2}(\rho')>\delta$ and $\varphi_{i-2}(\rho')>\tau_{\alpha_1,\ldots,\alpha_{k-i},\omega\beta}$. The choice $\rho'=\varphi_{i-1}(\rho+\beta)$ satisfies both these conditions. Hence, by the claim on $i-1$, we can bound the right-hand side (\ref{eq_ends_in_omegagamma}) and obtain
    \begin{align}
        \tau_{\alpha_1,\ldots,\alpha_{k-i},\omega\beta+1}\le\varphi_{i-2}(\rho'+\gamma)&\le\varphi_{i-2}(\varphi_{i-1}(\rho+\beta)+\delta\otimes\tau_{\alpha_1,\ldots,\alpha_{k-i},\omega\beta}^{[k-i+1]}+1)\nonumber\\&\le\varphi_{i-2}(\varphi_{i-1}(\rho+\beta)^\omega+1)\nonumber\\&\le\varphi_{i-2}(\varphi_{i-2}(\varphi_{i-1}(\rho+\beta)+1)).\label{eq_bd_omegabetap1}
    \end{align}
    Similarly, by (\ref{eq_num_subintervals}) and (\ref{eq_sigma_delta}), we have
    \begin{equation}\label{eq_ends_in_omegagammap1}
    \tau_{\alpha_1,\ldots,\alpha_{k-i},\omega\beta+2}\le\tau_{\alpha_1,\ldots,\alpha_{k-i},\omega\beta+1,\omega\gamma}
    \end{equation}
    for $\gamma=\delta\otimes\tau_{\alpha_1,\ldots,\alpha_{k-i},\omega\beta+1}^{[k-i+1]}+1$. Again, we want to bound the right-hand side of (\ref{eq_ends_in_omegagammap1}) by applying the claim on $i-1$. This time we choose $\rho'$ to be the right-hand side of (\ref{eq_bd_omegabetap1}). Applying the claim on $i-1$, we obtain
    \begin{align*}
        \tau_{\alpha_1,\ldots,\alpha_k-i,\omega\beta+2}\le \varphi_{i-2}(\rho'+\gamma)&\le \varphi_{i-2}(\varphi_{i-2}(\varphi_{i-2}(\varphi_{i-1}(\rho+\beta)+1))^\omega+1)\\&\le\varphi_{i-2}^{(4)}(\varphi_{i-1}(\rho+\beta)+1).
    \end{align*}
    In general, by induction on $j$ we get $\tau_{\alpha_1,\ldots,\alpha_{k-i},\omega\beta+j}\le\varphi_{i-2}^{(2j)}(\varphi_{i-1}(\rho+\beta)+1)$. Hence,
    \begin{equation*}
        \tau_{\alpha_1,\ldots,\alpha_{k-i},\omega(\beta+1)}=\tau_{\alpha_1,\ldots,\alpha_{k-i},\omega\beta+\omega}\le\sup_{j\in\N}\varphi_{i-2}^{(j)}(\varphi_{i-1}(\rho+\beta)+1)=\varphi_{i-1}(\rho+\beta+1),
    \end{equation*}
    as desired. The case where $\beta$ is a limit ordinal follows by continuity.
\end{proof}

\begin{Lemma}\label{lem_varphi_k}
    Let $\gamma$ be sufficiently large such that $\delta<\varphi_{k-1}(\gamma)$. Then
    \begin{equation*}
        \tau_{\omega\beta}\le\varphi_{k-1}(\gamma+\beta)\qquad\text{for every $\beta$.}
    \end{equation*}
\end{Lemma}

\begin{proof}
    By induction on $\beta$. The case $\beta=0$ follows since $\tau_0=o(f)\le \delta < \varphi_{k-1}(\gamma)$. Now suppose by induction that $\tau_{\omega\beta}\le\varphi_{k-1}(\gamma+\beta)$. By (\ref{eq_num_subintervals}) and (\ref{eq_sigma_delta}), the interval $I_{\omega\beta}$ is partitioned into at most $\omega\beta'$ subintervals for $\beta'=\delta\otimes\tau_{\omega\beta}+1$. Hence,
    \begin{equation*}
    \tau_{\omega\beta+1}\le\tau_{\omega\beta,\omega\beta'}.
    \end{equation*}
    Apply Lemma \ref{lem_varphi_i} with $i=k-1$ and $\alpha_1=\omega\beta$. We can take $\rho=\varphi_{k-1}(\gamma+\beta)$. Lemma \ref{lem_varphi_i} yields $\tau_{\omega\beta+1}\le \varphi_{k-2}(\rho+\beta')\le \varphi_{k-2}^{(2)}(\varphi_{k-1}(\gamma+\beta)+1)$. Similarly, we get $\tau_{\omega\beta+2}\le\varphi_{k-2}^{(4)}(\varphi_{k-1}(\gamma+\beta)+1)$, and, by induction on $j$, $\tau_{\omega\beta+j}\le\varphi_{k-2}^{(2j)}(\varphi_{k-1}(\gamma+\beta)+1)$. Thence, $\tau_{\omega(\beta+1)}\le\varphi_{k-1}(\gamma+\beta+1)$, as desired. The case where $\beta$ is a limit ordinal follows by continuity.
\end{proof}

Since the number of intervals $I_{\alpha_1}$ is at most $\omega\cdot(o(s)+1)$, we conclude that $o(M)\le \varphi_{k-1}(\gamma+o(s)+1)$, for $\gamma$ satisfying $\varphi_{k-1}(\gamma)>\delta$. Since $\varphi_{k-1}(\gamma)>\max\{o(f),o(s),o(g_1),\ldots,o(g_k)\}$ implies $\varphi_{k-1}(\gamma)>o(f)\otimes o(s)\otimes o(g_1)\otimes\cdots\otimes o(g_k)$, we are done.

\paragraph{Acknowledgements.} Thanks to the referees for their careful reading and helpful comments.

\bibliographystyle{plain}
\bibliography{refs.bib}{}

@article{altman,
author = {Altman, Harry J.},
title = {Intermediate arithmetic operations on ordinal numbers},
journal = {Mathematical Logic Quarterly},
volume = {63},
number = {3-4},
pages = {228-242},
doi = {https://doi.org/10.1002/malq.201600006},
year = {2017}
}

@article{ramsey,
author = "F. P. Ramsey",
title = "On a problem of formal logic",
journal = "Proc. Lond. Math. Soc.",
volume = "S2--30",
pages = "264--286",
year = "1930",
}

@article{blass_gurevich,
author = {Blass, Andreas and Gurevich, Yuri},
title = {Program termination and well partial orderings},
year = {2008},
issue_date = {June 2008},
publisher = {Association for Computing Machinery},
address = {New York, NY, USA},
volume = {9},
number = {3},
issn = {1529-3785},
url = {https://doi.org/10.1145/1352582.1352586},
doi = {10.1145/1352582.1352586},
journal = {ACM Trans. Comput. Logic},
articleno = {18},
numpages = {26}
}

@article{ENX,
    author  = "Jeff Erickson and Gabriel Nivasch and Junyan Xu",
    title   = "Fusible numbers and {P}eano {A}rithmetic",
    year    = "2022",
    journal = "Logical Methods in Computer Science",
    volume  = "18",
    number  = "3",
    numpages = "26",
}

@article{BNP,
title = {Generalized fusible numbers and their ordinals},
journal = {Annals of Pure and Applied Logic},
volume = {175},
number = {1, Part A},
numpages = {25},
year = {2024},
issn = {0168-0072},
doi = {https://doi.org/10.1016/j.apal.2023.103355},
author = {Alexander I. Bufetov and Gabriel Nivasch and Fedor Pakhomov},

}

@misc{Eri,
      title={Fusible Numbers},
      author={Jeff Erickson},
      note={Online presentation slides. \texttt{https://www.mathpuzzle.com/\allowbreak fusible.pdf}}
      }

@misc{Xu12,
      title={Survey on fusible numbers}, 
      author={Junyan Xu},
      year=2012,
      eprint={1202.5614},
      archivePrefix={arXiv},
      note={Unpublished manuscript. arXiv e-prints, math.CO, 1202.5614},
      primaryClass={math.CO}
}

@article{L.Clark,
 author = {Pete L. Clark},
 journal = {Mathematics Magazine},
 number = {2},
 pages = {136--150},
 title = {The Instructor’s Guide to Real Induction},
 volume = {92},
 year = {2019}
}

@article{RATHJEN199349,
title = {Proof-theoretic investigations on {K}ruskal's theorem},
journal = {Annals of Pure and Applied Logic},
volume = {60},
number = {1},
pages = {49-88},
year = {1993},
issn = {0168-0072},
doi = {https://doi.org/10.1016/0168-0072(93)90192-G},
url = {https://www.sciencedirect.com/science/article/pii/016800729390192G},
author = {Michael Rathjen and Andreas Weiermann}
}

@article{dejongh_parikh,
author = "Dick H. J. de Jongh and Rohit Parikh",
title = "Well-partial orderings and hierarchies",
journal = "Indagationes Mathematicae",
volume = "39",
year = "1977",
pages = "195--206",
}

\end{document}